\pdfoutput=1
\documentclass[letterpaper]{amsart}

\usepackage[utf8]{inputenc}
\usepackage{lmodern}
\usepackage[T1]{fontenc}
\usepackage{amssymb}
\usepackage{enumitem}
\usepackage{mathtools}

\usepackage{tikz-cd}
\usepackage{hyperref}
\usepackage{cleveref}

\tikzcdset{arrow style=Latin Modern}
\mathtoolsset{mathic}

\newlist{enumarabic}{enumerate}{1}
\setlist[enumarabic]{font=\normalfont,label=(\arabic*),leftmargin=0.3in}
\newlist{enumroman}{enumerate}{1}
\setlist[enumroman]{font=\normalfont,label=(\roman*),leftmargin=0.3in}

\theoremstyle{plain}
\newtheorem{theorem}{Theorem}[section]

\newtheorem{lemma}[theorem]{Lemma}

\theoremstyle{definition}

\newtheorem{remark}[theorem]{Remark}

\theoremstyle{remark}
\newtheorem*{acknowledgements}{Acknowledgements}

\numberwithin{equation}{section}

\let\newterm\emph

\def\arxiv#1{\href{http://arxiv.org/abs/#1}{\texttt{arXiv:#1}}}
\def\doi#1{\href{http://dx.doi.org/#1}{doi:#1}}

\def\cf{\emph{cf.}}

\let\epsilon\varepsilon
\let\emptyset\varnothing
\let\setminus\smallsetminus

\DeclareMathAlphabet\mathbfit{OML}{cmm}{b}{it}

\def\N{\mathbb{N}}
\def\Z{\mathbb{Z}}
\def\R{\mathbb{R}}

\def\kk{\Z}

\def\deg#1{|#1|}

\DeclareMathOperator{\Tor}{Tor}

\DeclareMathOperator{\supp}{supp}

\let\shuffle\nabla

\def\SRC#1{\kk\langle#1\rangle}
\def\restr#1#2{#1_{|#2}}

\def\ZK{\mathcal{Z}}
\def\RZK{\ZK_{\R}}

\def\hZK{\hat\ZK}
\def\RhZK{\hZK_{\R}}
\def\WK{\mathcal{S}}

\def\hB{\hat B}

\def\Hc{H_{c}}

\def\CB{C\,}
\def\CU{\mathcal{C}\,}

\def\AC{\mathbfit{A}}
\def\BC{\mathbfit{B}}
\def\KC{\mathbfit{K}}
\def\LC{\mathbfit{L}}
\def\RC{\mathbfit{R}}
\def\AR{A}
\def\BR{B}

\def\LR{L}
\def\RR{R}
\def\BCt{\boldsymbol{\tilde B}}

\def\PsiR{\Psi_{\R}}

\begin{document}

\title{Dga models for moment-angle complexes}
\author{Matthias Franz}
\thanks{The author was supported by an NSERC Discovery Grant.}
\address{Department of Mathematics, University of Western Ontario,
  London, Ont.\ N6A\;5B7, Canada}
\email{mfranz@uwo.ca}


\begin{abstract}
  A dga model for the integral singular cochains on a moment-angle complex
  is given by the twisted tensor product of the corresponding Stanley--Reisner ring
  and an exterior algebra. We present a short proof of this fact and extend it
  to real moment-angle complexes. We also compare various descriptions
  of the cohomology rings of these spaces, including one stated without proof
  by Gitler and López de Medrano.
\end{abstract}

\maketitle

\section{Introduction}

Let \(\Sigma\) be a simplicial complex on the set~\([m]=\{1,\dots,m\}\),
containing the empty simplex~\(\emptyset\) and possibly having ghost vertices, and let
\begin{equation}
  \ZK(\Sigma) = \ZK_{\Sigma}(D^{2},S^{1}) = \bigcup_{\sigma\in\Sigma} (D^{2},S^{1})^{\sigma}
  \subset (D^{2})^{m}
\end{equation}
be the associated moment-angle complex, where
\begin{equation}
  (D^{2},S^{1})^{\sigma}
  = \bigl\{\,(z_{1},\dots,z_{m})\in (D^{2})^{m} \bigm| \text{\(z_{i}\in S^{1}\) if \(i\notin\sigma\)}\,\bigr\}.
\end{equation}
Moment-angle complexes play a central role in toric topology, see~\cite{BuchstaberPanov:2015}.
Replacing \((D^{2},S^{1})\) by~\((D^{n},S^{n-1})\) for any~\(n\ge1\) gives generalized moment-angle complexes.
Taking arbitrary CW~pairs leads to polyhedral products, which have gained a lot of attention in homotopy theory recently,
see \cite{BahriEtAl:2020} for a survey.

The moment-angle complex~\(\ZK(\Sigma)\) is homotopy-equivalent to the complement
of a complex coordinate subspace arrangement, which is a smooth toric variety.
The integral cohomology ring of~\(\ZK(\Sigma)\) was computed by the author~\cite[Sec.~4]{Franz:2003a}
(using the language of toric varieties)
and shortly afterwards by Baskakov--Buchstaber--Panov~\cite{BaskakovBuchstaberPanov:2004}.\footnote{%
  The argument appearing in~\cite[Thm.~7.7]{BuchstaberPanov:2002} and earlier publications by the same authors is incorrect,
  compare~\cite[Sec.~1]{Franz:2018b}.}
The result is an isomorphism of graded rings
\begin{equation}
  \label{eq:iso-HZK-Tor}
  H^{*}(\ZK(\Sigma)) = \Tor_{\RC}(\Z,\Z[\Sigma]),
\end{equation}
where \(\RC=\Z[t_{1},\dots,t_{m}]\) and \(\Z[\Sigma]\) is the Stanley--Reisner ring of~\(\Sigma\)
with generators~\(t_{1}\),~\dots,~\(t_{m}\) of degree~\(2\).
Taking the Koszul resolution of~\(\Z\) over~\(\RC\),
one can describe the ring~\eqref{eq:iso-HZK-Tor} as the cohomology of the commutative
differential graded algebra (cdga)
\begin{equation}
  \label{eq:def-A}
  \AC(\Sigma) = \kk[\Sigma] \otimes \bigwedge(s_{1},\dots,s_{m}),
  \qquad
  d\,s_{i} = t_{i},
  \quad
  d\,t_{i} =0
\end{equation}
for~\(i\in[m]\), where each~\(s_{i}\) has degree~\(1\).
Dividing out out all squares~\(t_{i}^{2}\) as well as all terms~\(s_{i}\,t_{i}\),
one obtains a quasi-isomorphic dga~\(\BC(\Sigma)\). 
As a cdga, \(\BC(\Sigma)\) is generated by the~\(s_{i}\) and~\(t_{i}=d\,s_{i}\) and has the relations \(s_{i}\,t_{i}=t_{i}\,t_{i}=0\) for~\(i\in[m]\)
as well as \(t_{i_{1}}\cdots t_{i_{k}}=0\) whenever \(\{i_{1},\dots,i_{k}\}\notin\Sigma\).

\begin{theorem}
  \label{thm:main}
  The singular cochain algebra~\(C^{*}(\ZK(\Sigma))\)
  is quasi-iso\-mor\-phic to the dgas~\(\AC(\Sigma)\) and~\(\BC(\Sigma)\),
  naturally with respect to inclusions of subcomplexes.
\end{theorem}

Recall that a dga~\(A\) is called an (integral) \newterm{dga model} for a space~\(X\) if \(A\) can be connected to~\(C^{*}(X)\)
via a zigzag of dga quasi-isomorphisms. In this language,
\Cref{thm:main} asserts that both~\(\AC(\Sigma)\) and~\(\BC(\Sigma)\) are dga models for~\(\ZK(\Sigma)\).

That \(C^{*}(\ZK(\Sigma))\) and~\(\AC(\Sigma)\) are quasi-isomorphic
is already implicit in the author's computation of~\(H^{*}(\ZK(\Sigma))\),
see~\cite[Sec.~4]{Franz:2003a}.
A different proof has recently been obtained by the author as a byproduct
of his work on the cohomology rings of partial quotients
of moment-angle complexes \cite[Prop.~6.1]{Franz:2018b}.
As remarked there, this result answers a question
posed by Berglund~\cite[Question~5]{Berglund:2010},
which was exactly whether \(\AC(\Sigma)\) is a dga model for~\(\ZK(\Sigma)\).
The aim of the present note is to give a much shorter proof for this model.
Like Baskakov--Buchstaber--Panov's calculation it is based on the dga~\(\BC(\Sigma)\).
The rational versions of~\(\AC(\Sigma)\) and~\(\BC(\Sigma)\) are
(analogously defined) cdga models for the polynomial differential forms
on~\(\ZK(\Sigma)\) by a result of Panov--Ray~\cite[Thm.~6.2]{PanovRay:2008}.

The proof of \Cref{thm:main} appears in the following section
and an adaptation to real moment-angle complexes in \Cref{sec:real}.
In the final section we relate the resulting cup product formulas for real and complex
moment-angle complexes with others appearing in the literature.
We in particular provide a proof that has been missing so far for a product formula
stated by Gitler and López de Medrano~\cite{GitlerLopezDeMedrano:2013}.

\begin{acknowledgements}
  I thank Don Stanley for his questions about dga models and Santiago López de Medrano for stimulating discussions.
  I also thank the organizers of the ``Thematic Program on Toric Topology and Polyhedral Products''
  at the Fields Institute for creating an environment favourable to research
  and an anonymous referee for comments that helped to improve the presentation of the paper.
\end{acknowledgements}

\section{Proof of Theorem~\ref{thm:main}}

We will obtain \Cref{thm:main} by dualizing the analogous homological result.
To state the latter, we need to introduce some terminology.
As already done in \Cref{thm:main}, we write \(C(-)\) and~\(C^{*}(-)\)
for \emph{normalized} singular (co)chains with integral coefficients.

Recall that the normalized singular chain complex of a space~\(X\)
is obtained from the usual non-normalized one by dividing out the subcomplex
of degenerate simplices. A singular \(n\)-simplex is called degenerated
if it factors through an \((n-1)\)-dimensional one via a map~\(\Delta^{n}\to\Delta^{n-1}\)
between standard simplices that in barycentric coordinates is given
by~\((t_{0},\dots,t_{n})\mapsto (t_{0},\dots,t_{i}+t_{i+1},\dots,t_{n})\)
for some~\(0\le i<n\). Projecting non-normalized to normalized singular chains
is a homotopy equivalence, compare~\cite[Sec.~VIII.6]{MacLane:1967}.
The normalized singular cochain complex~\(C^{*}(X)\) is the dual of~\(C(X)\) with differential
\begin{equation}
  \label{eq:def-d-dual}
  (d\,\gamma)(x) = - (-1)^{\deg{\gamma}}\,\gamma(d\,x)
\end{equation}
for~\(\gamma\in C^{n}(X)\) and a singular \((n+1)\)-simplex~\(x\).

The chain complex~\(C(X)\) is a differential graded coalgebra (dgc)
with diagonal and augmentation given by
\begin{equation}
  \Delta x = \sum_{k=0}^{n} x(0 \dots k) \otimes x(k \dots n)
  \qquad\text{and}\qquad
  \epsilon(x) = \begin{cases}
    1 & \text{if \(n=0\),} \\
    0 & \text{otherwise}
  \end{cases}
\end{equation}
for an \(n\)-simplex~\(x\). Here \(x(k_{1}\dots k_{2})\) denotes the face of~\(x\) with vertices~\(k_{1}\),~\dots,~\(k_{2}\).

We also recall part of the Eilenberg--Zilber theorem, compare~\cite[Sec.~17]{EilenbergMoore:1966}.
Given two spaces~\(X\) and~\(Y\), the shuffle map
\begin{equation}
  \shuffle\colon C(X)\otimes C(Y) \to C(X\times Y)
\end{equation}
is a homotopy equivalences of complexes, natural in~\(X\) and~\(Y\).
It moreover is associative and a morphism of dgcs.
Hence for any spaces~\(X_{1}\),~\dots,~\(X_{m}\) we have a natural quasi-isomorphism of dgcs
\begin{equation}
  C(X_{1})\otimes\dots\otimes C(X_{m}) \to C\bigl(X_{1}\times\dots\times X_{m}\bigr),
\end{equation}
again denoted by~\(\shuffle\).

Let \(\SRC{\Sigma}\) be the Stanley--Reisner coalgebra of~\(\Sigma\) dual to~\(\kk[\Sigma]\),
\cf~\cite[Sec.~8.2]{BuchstaberPanov:2015}.
The  canonical basis for~\(\SRC{\Sigma}\), considered as a \(\Z\)-module, are the monomials~\(u_{\alpha}\)
indexed by allowed multi-indices~\(\alpha\in\N^{m}\).
A multi-index~\(\alpha\) is \newterm{allowed} if it is supported on some simplex in~\(\Sigma\), that is, if
\begin{equation}
  \supp\alpha \coloneqq \{\, i\in[m] \mid \alpha_{i} > 0\,\} \in \Sigma.
\end{equation}
The degree of~\(u_{\alpha}\) is \(2(\alpha_{1}+\dots+\alpha_{m})\).
The structure maps are given by
\begin{equation}
  \Delta u_{\alpha} = \sum_{\beta+\gamma=\alpha} u_{\beta}\otimes u_{\gamma},
  \qquad
  \epsilon(u_{\alpha}) =
  \begin{cases}
    1 & \text{if \(\alpha=0\),} \\
    0 & \text{otherwise.}
  \end{cases}
\end{equation}

We consider the tensor product of graded coalgebras
\begin{equation}
  \KC(\Sigma) =  \SRC{\Sigma} \otimes \bigwedge(v_{1},\dots,v_{m}),
\end{equation}
where each~\(v_{i}\) is primitive of degree~\(1\). We turn~\(\KC(\Sigma)\) into a dgc by defining
\begin{equation}
  d(u_{\alpha}\otimes v_{\tau}) = \sum_{\alpha_{i}>0} u_{\alpha-i}\otimes v_{i}\wedge v_{\tau}
\end{equation}
for allowed multi-indices~\(\alpha\in\N^{m}\) and~\(\tau\subset[m]\).
Here we have written \(\alpha-i\) for the multi-index
that is obtained from~\(\alpha\) by decreasing the \(i\)-th component by~\(1\)
as well as \(v_{\tau}=v_{i_{1}}\wedge\dots\wedge v_{i_{k}}\) if~\(\tau=\{i_{1}<\dots<i_{k}\}\).
For~\(\sigma\in\Sigma\) we also write \(u_{\sigma}=u_{\alpha}\)
where \(\alpha\) is the indicator function of~\(\sigma\subset[m]\),
\begin{equation}
  \alpha_{i} =
  \begin{cases}
    1 & \text{if \(i\in\sigma\),} \\
    0 & \text{if \(i\notin\sigma\),}
  \end{cases}
\end{equation}
and we use the abbreviation~\(u_{\emptyset}=v_{\emptyset}=u_{\emptyset}\otimes v_{\emptyset}=1\).

Let \(\LC(\Sigma)\) be the sub-dgc of~\(\KC(\Sigma)\) spanned by all elements~\(u_{\sigma}\otimes v_{\tau}\)
with disjoint subsets~\(\sigma\in\Sigma\) and~\(\tau\subset[m]\).
The dual of~\(\KC(\Sigma)\) is the dga~\(\AC(\Sigma)\), and that of~\(\LC(\Sigma)\) is \(\BC(\Sigma)\).

\begin{theorem}
  \label{thm:main-hom}
  The dgcs \(C(\ZK(\Sigma))\),~\(\KC(\Sigma)\) and~\(\LC(\Sigma)\) are quasi-isomorphic,
  naturally with respect to inclusions of subcomplexes.
\end{theorem}

The proof is given in the remainder of this section.
Applying the universal coefficient theorem for cohomology then establishes \Cref{thm:main}.

\medbreak

The following two observations are immediate.
We write \(\restr{\Sigma}{i}\) for the restriction of~\(\Sigma\) to the single vertex~\(i\in[m]\).
It contains either the empty simplex only or additionally the \(0\)-simplex~\(\{i\}\).

\begin{lemma}
  \label{thm:KL-tensor}
  For any~\(\sigma\in\Sigma\) there are canonical isomorphisms of dgcs
  \begin{equation*}
    \KC(\sigma) = \bigotimes_{i=1}^{m} \KC(\restr{\sigma}{i}),
    \qquad
    \LC(\sigma) = \bigotimes_{i=1}^{m} \LC(\restr{\sigma}{i}).
  \end{equation*}
\end{lemma}

\begin{lemma}
  \label{thm:KL-MV}
  Let \(\Sigma_{1}\),~\(\Sigma_{2}\) be subcomplexes of~\(\Sigma\).
  There are short exact sequences
  \begin{gather*}
    0 \longrightarrow \KC(\Sigma_{1}\cap\Sigma_{2})
    \longrightarrow \KC(\Sigma_{1})\oplus \KC(\Sigma_{2})
    \longrightarrow \KC(\Sigma_{1}\cup\Sigma_{2})
    \longrightarrow 0,
    \\
    0 \longrightarrow \LC(\Sigma_{1}\cap\Sigma_{2})
    \longrightarrow \LC(\Sigma_{1})\oplus \LC(\Sigma_{2})
    \longrightarrow \LC(\Sigma_{1}\cup\Sigma_{2})
    \longrightarrow 0.
  \end{gather*}
\end{lemma}

Let \(y\) be the usual parametrization of~\(S^{1}\), considered as a singular \(1\)-simplex.
Choose a singular \(2\)-simplex~\(x\) in~\(D^{2}\) that
restricts to~\(y\) on the edge~\((12)\) and maps the other two edges~\((01)\) and~\((02)\) to the identity element~\(e\in S^{1}\). Then
\begin{align}
  \label{eq:y}
  d\,y &= 0, & \Delta y &= y\otimes e + e\otimes y, \\
  \label{eq:x}
  d\,x &= x(12)-x(02)+x(01) &\qquad\quad  \Delta x &= x\otimes e + x(01)\otimes x(12) + e\otimes x \\
  \notag
  &= y, &  &= x\otimes e + e\otimes x.
\end{align}
Here \(x(01)\) and~\(x(02)\) drop out because they are degenerate. For this
to hold it is crucial that we work with normalized chains.

We use the singular simplices~\(x\) and~\(y\) to define a dgc map
\begin{equation}
  \Psi(\Sigma)\colon \LC(\Sigma)\to C(\ZK(\Sigma)).
\end{equation}
For~\(m=1\) we map~\(u_{1}\mapsto x\),
\(v_{1}\mapsto y\) and~\(1\mapsto e\in C(S^{1})\);
this is well-defined by~\eqref{eq:y} and~\eqref{eq:x}. For~\(m>1\) and~\(\sigma\subset[m]\) we set
\begin{multline}
  \label{eq:def-Psi-sigma}
  \qquad
  \Psi(\sigma)\colon
  \LC(\sigma) = \bigotimes_{i=1}^{m} \LC(\restr{\sigma}{i})
  \xrightarrow{\bigotimes\Psi(\restr{\sigma}{i})} \bigotimes_{i=1}^{m} C(\ZK(\restr{\sigma}{i})) \\
  \xrightarrow{\;\;\shuffle\;\;} C\bigl(\ZK(\restr{\sigma}{1})\times\dots\times \ZK(\restr{\sigma}{m})\bigr)
  = C(\ZK(\sigma)),
  \qquad
\end{multline}
using \Cref{thm:KL-tensor} and the fact that the shuffle map is a morphism of dgcs.
In the general case \(\Psi(\Sigma)\) is determined by imposing
naturality with respect to inclusions of subcomplexes.
In other words, \(\Psi(\Sigma)\) agrees on~\(\LC(\sigma)\subset\LC(\Sigma)\) with~\(\Psi(\sigma)\),
followed by the inclusion~\(C(\ZK(\sigma))\hookrightarrow C(\ZK(\Sigma))\).

\begin{lemma}
  \label{thm:LC-CZK-quiso}
  The map~\(\Psi(\Sigma)\) is a quasi-isomorphism of dgcs.
\end{lemma}

\begin{proof}
  The case~\(m=1\) is settled by a direct verification; we therefore assume \(m>1\)
  and proceed by induction on the size of~\(\Sigma\).
  If \(\Sigma\) has a single maximal simplex~\(\sigma\), then \(\Psi(\Sigma)=\Psi(\sigma)\)
  is a quasi-isomorphism because so are the shuffle map and, by the Künneth theorem,
  the tensor product of the maps~\(\Psi(\restr{\sigma}{i})\).

  Otherwise we can split \(\Sigma\) up into two smaller complexes~\(\Sigma_{1}\) and~\(\Sigma_{2}\)
  with intersection~\(\Sigma_{12}=\Sigma_{1}\cap\Sigma_{2}\).
  The naturality of~\(\Psi\) gives us a map between the long exact sequence corresponding
  to the short exact sequence for~\(\LC\) from \Cref{thm:KL-MV} and the Mayer--Vietoris sequence
  for the CW~complex~\(\ZK(\Sigma)=\ZK(\Sigma_{1})\cup\ZK(\Sigma_{2})\),
  \begin{equation}
    \footnotesize
    \begin{tikzcd}[column sep=1em]
      \cdots \arrow{r} & H_{*+1}(\LC(\Sigma)) \arrow{d}{\Psi_{*+1}(\Sigma)} \arrow{r}
      & H_{*}(\LC(\Sigma_{12})) \arrow{d}{\Psi_{*}(\Sigma_{12})} \arrow{r}
      & H_{*}(\LC(\Sigma_{1})) \oplus H_{*}(\LC(\Sigma_{2})) \arrow{d}{(\Psi_{*}(\Sigma_{1}),\Psi_{*}(\Sigma_{2}))} \arrow{r}
      & H_{*}(\LC(\Sigma)) \arrow{d}{\Psi_{*}(\Sigma)} \arrow{r} & \cdots \\
      \cdots \arrow{r} & H_{*+1}(\ZK(\Sigma)) \arrow{r}
      & H_{*}(\ZK(\Sigma_{12})) \arrow{r}
      & H_{*}(\ZK(\Sigma_{1})) \oplus H_{*}(\ZK(\Sigma_{2})) \arrow{r} 
      & H_{*}(\ZK(\Sigma)) \arrow{r} & \cdots \mathrlap{\;,}
    \end{tikzcd}
  \end{equation}
  where we have written \(\Psi_{*}(\Sigma)\) instead of~\(H_{*}(\Psi(\Sigma))\) etc.
  The five lemma together with induction implies
  that \(\Psi_{*}(\Sigma)\) is an isomorphism.
\end{proof}

An analogous argument shows that the inclusion map
\begin{equation}
  \label{eq:LC-KC-quiso}
  \LC(\Sigma)\hookrightarrow\KC(\Sigma)
\end{equation}
is a quasi-isomorphism of dgcs.\footnote{%
  It is also a homotopy equivalence of complexes.
  See~\cite[Lemma~7.10]{BuchstaberPanov:2002} or~\cite[Lemma~3.2.6]{BuchstaberPanov:2015} for an explicit homotopy inverse.}
If \(\Sigma\) has a single maximal simplex, we combine \Cref{thm:KL-tensor} with the Künneth theorem.
Otherwise we again write \(\Sigma=\Sigma_{1}\cup\Sigma_{2}\) and compare the long exact sequences associated
to both short exact sequences in \Cref{thm:KL-MV}.

This completes the proof of \Cref{thm:main-hom}.

\begin{remark}
  \label{rem:generalized}
  Theorems~\ref{thm:main} and~\ref{thm:main-hom} remain valid for all generalized moment-angle complexes~\(\ZK_{\Sigma}(D^{n},S^{n-1})\) with even~\(n\ge 2\),
  up to the obvious degree shifts. In particular, the generators~\(s_{i}\) and~\(t_{i}\) in~\eqref{eq:def-A} are now of degrees~\(\deg{s_{i}}=n-1\) and~\(\deg{t_{i}}=n\).
  The singular \(n\)-simplex~\(x\) is obtained by collapsing all but the last facet of the standard \(n\)-simplex to a point, and \(y\) is this last facet.
  
  If \(n\ge3\) is odd, then \(\deg{y}\) is even and \(\deg{x}\) is odd. Proceeding as before, we get a quasi-isomorphism
  between~\(C^{*}(\ZK_{\Sigma}(D^{n},S^{n-1}))\) and the cdga~\(\BCt(\Sigma)\) with generators~\(s_{i}\) of degree~\(n-1\) and~\(t_{i}=d\,s_{i}\) of degree~\(n\)
  as well as relations
  \begin{equation}
    s_{i}\,s_{i} = s_{i}\,t_{i} = 0,
    \qquad\text{and}\qquad
    t_{i_{1}}\cdots t_{i_{k}}=0 \quad \text{if \(\{i_{1},\dots,i_{k}\}\notin\Sigma\).}
  \end{equation}
  Note that the Stanley--Reisner relations are monomial and therefore independent of the order
  of the anticommuting variables~\(t_{i}\).
  
  In general, such a quasi-isomorphism does not hold for the case~\(n=1\), which we treat in the following section.
\end{remark}

\section{Real moment-angle complexes}
\label{sec:real}

It is not difficult to adapt our approach to real moment-angle complexes
\begin{equation}
  \RZK(\Sigma) = \ZK_{\Sigma}(D^{1},S^{0}) \subset (D^{1})^{m}.
\end{equation}

We start with the homological setting.
As a chain complex, we define the analogue~\(\LR(\Sigma)\) of~\(\LC(\Sigma)\) as before, except that now the degrees are
\(\deg{u_{i}} = 1\) and~\(\deg{v_{i}}=0\) for all~\(i\in[m]\).

Let us consider the case~\(m=1\) first. Writing \(u=u_{1}\) and~\(v=v_{1}\),
we turn \(\LR(\Sigma)\) into a dgc via the augmentation~\(\epsilon(v)=\epsilon(u)=0\) and the diagonal
\begin{align}
  \Delta v &= v\otimes 1 + 1\otimes v + v\otimes v, \\
  \Delta u &= u\otimes 1 + 1\otimes u + u\otimes v.
\end{align}

Let \(x\) be the canonical path from~\(e=1\) to~\(g=-1\in S^{0}\), considered as a singular \(1\)-simplex in~\(D^{1}=[-1,1]\),
and let \(y=g-e\). Then
\begin{align}
  \label{eq:d-x-y}
  &\qquad d\, x = y, \qquad\qquad d\,y = 0, \\
  \label{eq:Delta-y}
  \Delta y &= g\otimes g - e\otimes e = y\otimes e + e\otimes y + y\otimes y, \\
  \label{eq:Delta-x}
  \Delta x &= x\otimes g + e\otimes x = x\otimes e + e\otimes x + x\otimes y.
\end{align}
Given that for~\(m=1\) we either have \(\Sigma=\{\emptyset\}\) or \(\Sigma=\{\emptyset,\{1\}\}\),
one verifies directly
that the map
\begin{equation}
  \label{eq:map-L-C-real}
  \PsiR(\Sigma)\colon \LR(\Sigma) \to C(\RZK(\Sigma)),
  \qquad
  1 \mapsto e, \quad v \mapsto y, \quad u \mapsto x
\end{equation}
is a quasi-isomorphism of dgcs. (Since this map is injective, one can also use it
to justify that \(\LR(\Sigma)\) with the diagonal given above is indeed a dgc.)

For~\(m>1\) we again proceed exactly as before.
We use the isomorphism of complexes
\begin{equation}
  \LR(\sigma) = \bigotimes_{i=1}^{m} \LR(\restr{\sigma}{i})
\end{equation}
analogous to \Cref{thm:KL-tensor} to define a dgc structure on~\(\LR(\sigma)\)
for each simplex~\(\sigma\in\Sigma\) as well as a dgc map~\(\LR(\sigma)\to C(\RZK(\sigma))\).
Then we extend both the dgc structure and the map to the colimit~\(\LR(\Sigma)\) of the~\(\LR(\sigma)\) over all~\(\sigma\in\Sigma\).
The same proof as for \Cref{thm:LC-CZK-quiso} shows that the resulting map
\begin{equation}
  \PsiR(\Sigma)\colon\LR(\Sigma)\to C(\RZK(\Sigma))
\end{equation}
is a quasi-isomorphism.
In this step one uses the obvious analogue of \Cref{thm:KL-MV} for~\(\LR\) instead of~\(\LC\).

We thus have established the following.

\begin{lemma}
  The map~\(\PsiR(\Sigma)\) is a quasi-isomorphism of dgcs.
\end{lemma}

\medskip

We now turn to cohomology.
The dual of the dgc~\(\LR(\Sigma)\) is the dga~\(\BR(\Sigma)\) with generators~\(s_{i}\) of degree~\(0\)
and \(t_{i}\) of degree~\(1\) satisfying the relations
\begin{gather}
  d\,s_{i}=-t_{i}, \qquad d\,t_{i} = 0, \\
  \label{eq:relations-s-t}
  s_{i}\,s_{i} = s_{i}, \qquad t_{i}\,s_{i} = t_{i}, \qquad s_{i}\,t_{i} = 0, \qquad t_{i}\,t_{i} = 0,
  \qquad \prod_{j\in\sigma} t_{j} = 0
\end{gather}
for any~\(i\in[m]\) and~\(\sigma\notin\Sigma\)
plus the rule that variables corresponding to \emph{distinct} subscripts commute in the graded sense.

The minus sign in~\(d\,s_{i} = t_{i}\) comes from the general definition
of the differential on the dual of a complex,
\cf~\eqref{eq:def-d-dual} and~\cite[eq.~(II.3.1)]{MacLane:1967}.
For~\(m=1\) we have
\begin{equation}
  (d\,s)(x) = -(-1)^{\deg{s}}\,s(d\,x) = - s(y) = -1 = -t(x).
\end{equation}
The minus sign could be removed by replacing \(t_{i}\) with~\(-t_{i}\), that is, by mapping \(u\) to~\(-x\).
The minus sign does not appear in~\cite[p.~512]{Cai:2017}
because of a different sign convention for the dual complex.

We can sum up our discussion as follows.

\begin{theorem}
  \label{thm:real-B}
  There is a quasi-isomorphisms of dgas
  \begin{equation*}
    C^{*}(\RZK(\Sigma)) \to \BR(\Sigma),
  \end{equation*}
  natural with respect to inclusions of subcomplexes.
\end{theorem}

We in particular recover Cai's isomorphism of graded rings~\cite[Secs.~3~\&~4]{Cai:2017}
\begin{equation}
  \label{eq:HRKZ-HBSigma}
  H^{*}(\RZK(\Sigma)) = H^{*}(\BR(\Sigma)).
\end{equation}
In fact, our proof shares some similarities with Cai's.
This would be even more so
if we worked with cubical singular chains, compare~\cite{Massey:1980}.
We also remark that in the case of real moment-angle complexes
it is not necessary to pass to normalized (singular) chains.
(The shuffle map is a morphism of dgcs for non-normalized chains already,
and the formulas~\eqref{eq:d-x-y}--\eqref{eq:Delta-x} do not need normalization, either.)

We discuss the dga~\(\AR(\Sigma)\) analogous to~\(\AC(\Sigma)\) only for coefficients in~\(\Z_{2}\).
General coefficients are considered in~\cite{Franz:realtorprod}
where \Cref{thm:real-B} is furthermore extended to real toric spaces, that is,
to quotients of~\(\RZK(\Sigma)\) by freely acting subgroups of~\((\Z_{2})^{m}\).

In characteristic~\(2\), the dga~\(\AR(\Sigma)\)
has the same generators~\(s_{i}\) and~\(t_{i}\) as~\(\BR(\Sigma)\) and the relations
\begin{gather}
  d\,s_{i}=t_{i}, \qquad d\,t_{i} = 0, \\
  \label{eq:relations-A-real}
  s_{i}\,s_{i} = s_{i}, \qquad t_{i}\,s_{i} = s_{i}\,t_{i} + t_{i},
  \qquad \prod_{j\in\sigma} t_{j} = 0
\end{gather}
for~\(i\in[m]\) and~\(\sigma\notin\Sigma\), again with the additional rule
that variables corresponding to different subscripts commute. 
Observe that the ideal generated by the relations~\eqref{eq:relations-A-real} is closed under the differential,
so that \(\AR(\Sigma)\) is a well-defined dga.
The projection map~\(\AR(\Sigma)\to \BR(\Sigma)\otimes\Z_{2}\)
is again obtained by dividing out the ideal generated by the products~\(s_{i}\,t_{i}\) and~\(t_{i}^{2}\) for all~\(i\in[m]\),
and it can be seen to be a quasi-isomorphism by an argument analogous to the one given before or to~\cite[Lemma~7.10]{BuchstaberPanov:2002}.

The Stanley--Reisner ring~\(\Z_{2}[\Sigma]\), now with generators of degree~\(1\),
is contained in~\(\AR(\Sigma)\) as a sub-dga (with trivial differential). Moreover,
if \(\Sigma=[m]\) is the full simplex, then \(\AR(\Sigma)\) is the Koszul resolution of~\(\Z_{2}\) over~\(\RR=\Z_{2}[t_{1},\dots,t_{m}]\).
In general, \(\AR(\Sigma)\) is the tensor product of this resolution and~\(\Z_{2}[\Sigma]\) over~\(\RR\),
which gives the additive isomorphism
\begin{equation}
  H^{*}(\RZK(\Sigma);\Z_{2}) = \Tor_{\RR}(\Z_{2},\Z_{2}[\Sigma]).
\end{equation}
It is not multiplicative for the canonical product on the torsion product,
as can be seen for~\(\Sigma=\{\emptyset\}\) already, \cf~\cite[Sec.~10.3]{Franz:2018b}:
In this case we have \(\Z_{2}[\Sigma]=\Z_{2}\), so that the torsion product is a strictly exterior algebra.
On the other hand, \(\RZK(\Sigma)=(S^{0})^{m}\) is a finite set of points.
Hence any cohomology class on~\(\RZK(\Sigma)\) is a \(\Z_{2}\)-values function on these points and therefore squares to itself.

\section{Comparison of several product formulas}
\label{sec:compare}

The aim of this section is to relate the product formula in the cohomology of a (complex) moment-angle complex
with Baskakov's formula~\cite{Baskakov:2002} and also the formula for real moment-angle complexes
with one claimed by Gitler and López de Medrano~\cite{GitlerLopezDeMedrano:2013}
as well as the one given by Bahri--Bendersky--Cohen--Gitler~\cite{BahriEtAl:2012}
for arbitrary polyhedral products.\footnote{%
  Another cup product formula
  has given by de~Longueville~\cite[Thm.~1.1]{DeLongueville:1999}
  for complements of real coordinate subspace arrangements.
  However, his formula is incorrect in general,
  see \cite[Sec.~3]{GitlerLopezDeMedrano:2013} for a counterexample.}
We note that another description for a class of polyhedral products including
all~\(\ZK_{\Sigma}(D^{n},S^{n-1})\) has been given by Zheng~\cite[Example~7.12]{Zheng:2016}.

We start with a variant of the generalized smash moment-angle complexes introduced in~\cite[Def.~2.2]{BahriEtAl:2012}.
For a closed subset~\(A\) of a compact Hausdorff space~\(X\) and a basepoint~\(*\in A\) we define the space
\begin{equation}
  \WK(X,A) = \bigl\{\, x \in \ZK(X,A) \bigm| \text{\(x_{i}=*\) for some~\(i\in[m]\)} \,\bigr\}
\end{equation}
and based on it the pair
\begin{equation}
  \hZK_{\Sigma}(X,A) = \bigl( \ZK_{\Sigma}(X,A), \WK_{\Sigma}(X,A) \bigr).
\end{equation}
We then have an isomorphism
\begin{equation}
  H^{*}(\hZK_{\Sigma}(X,A)) = \Hc^{*}\bigl(\ZK_{\Sigma}(X,A)\setminus\WK_{\Sigma}(X,A)\bigr)
  = \Hc^{*}\bigl(\ZK_{\Sigma}(X\setminus *,A\setminus *)\bigr),
\end{equation}
where \(\Hc^{*}(-)\) denotes cohomology with compact supports, \cf~\cite[Part~I]{Massey:1978}.

We now specialize to
\begin{equation}
  \label{eq:def-RhZK}
  \RhZK(\Sigma) = \hZK_{\Sigma}(D^{1},S^{0})
\end{equation}
(where the basepoint is~\(e=1\in S^{0}\)) and observe that
\begin{multline}
  \qquad
  \ZK_{\Sigma}(D^{1},S^{0})\setminus\WK_{\Sigma}(D^{1},S^{0}) =
  \ZK_{\Sigma}\bigl(D^{1}\setminus\{e\},S^{0}\setminus\{e\}\bigr) \\ = \ZK_{\Sigma}\bigl([-1,1),\{-1\} \bigr)
  \approx \ZK_{\Sigma}\bigl([0,\infty),\{0\} \bigr) = \CU \Sigma,
  \qquad
\end{multline}
where \(\CU \Sigma\) is the unbounded cone over the simplicial complex~\(\Sigma\).
We think of~\(\Sigma\) as embedded into the hyperplane of~\(\R^{m}\) with coordinate sum equal to~\(1\).

The analysis of~\(\RZK(\Sigma)\) in the preceding section carries over to the present case.
One simply ignores the element~\(e\in S^{0}\) and the counit~\(1\) in the cochain algebra.
(Recall that the cohomology with compact supports is a ring without unit in general.)
The result is as quasi-isomorphism between the relative cochain algebra~\(C^{*}(\RhZK(\Sigma))\)
and the multiplicatively closed subcomplex~\(\hB(\Sigma)\subset \BR(\Sigma)\) spanned by all \(m\)-fold products
\begin{equation}
  \label{eq:def-a1-am}
  a_{1}\cdots a_{m}
  \qquad
  \text{where each~\(a_{i}=s_{i}\) or~\(t_{i}\).}
\end{equation}
In particular, there is a multiplicative isomorphism
\begin{equation}
  \label{eq:iso-HcCU-HhB}
  \Hc^{*}(\CU \Sigma)  = H^{*}(\CB \Sigma,\Sigma) \cong H^{*}(\hB(\Sigma))
\end{equation}
where \(\CB \Sigma\) denotes the bounded cone over~\(\Sigma\) with base~\(\Sigma\).
Not surprisingly, \(\hB(\Sigma)\) does not have a unit unless~\(\Sigma=\{\emptyset\}\).

\medbreak

We now compare \(\hB(\Sigma)\) to the dgas~\(\BC(\Sigma)\) and~\(\BR(\Sigma)\) for complex and real moment-angle complexes, respectively.
In the complex case, we have a direct sum decomposition of complexes
\begin{equation}
  \BC^{*}(\Sigma) = \bigoplus_{\alpha\subset[m]} \hB^{*-\deg{\alpha}}(\Sigma_{\alpha})
\end{equation}
where \(\Sigma_{\alpha}\) is the full subcomplex of~\(\Sigma\) on the vertex set~\(\alpha\).
This gives Hochster's formula
\begin{equation}
  \label{eq:Hochster-ZK}
  H^{*}(\ZK(\Sigma)) = \bigoplus_{\alpha\subset[m]} \Hc^{*-\deg{\alpha}}(\CU \Sigma_{\alpha})
  = \bigoplus_{\alpha\subset[m]} \tilde H^{*-\deg{\alpha}-1}(\Sigma_{\alpha}),
\end{equation}
\cf~\cite[Thm.~3.2.7]{BuchstaberPanov:2002},
where we have used the additive isomorphism
\begin{equation}
  \Hc^{*}(\CU \Sigma) = \tilde H^{*-1}(\Sigma)
\end{equation}
between the reduced cohomology of the simplicial complex~\(\Sigma\) and
the cohomology with compact supports of the unbounded cone over it.
(Recall that \(\tilde H^{-1}(\emptyset)=\kk\).)

The additive isomorphism~\eqref{eq:Hochster-ZK} can be made multiplicative in the following way:
For~\(\alpha\cap\beta\ne\emptyset\), the product
\begin{equation}
  \Hc^{*}(\CU \Sigma_{\alpha}) \otimes \Hc^{*}(\CU \Sigma_{\beta}) \to \Hc^{*}(\CU \Sigma_{\alpha\cup\beta}),
\end{equation}
vanishes. For disjoint~\(\alpha\),~\(\beta\) we use the cross product via the composition
\begin{equation}
  \label{eq:cross-prod}
  \Hc^{*}(\CU \Sigma_{\alpha}) \otimes \Hc^{*}(\CU \Sigma_{\beta})
  \xrightarrow{\times} \Hc^{*}(\CU \Sigma_{\alpha}\times\CU \Sigma_{\beta})
  \xrightarrow{\iota^{*}} \Hc^{*}(\CU \Sigma_{\alpha\cup\beta}),
\end{equation}
where \(\iota\colon\CU \Sigma_{\alpha\cup\beta}\hookrightarrow\CU \Sigma_{\alpha}\times\CU \Sigma_{\beta}\) is the canonical inclusion.
This is Baskakov's formula~\cite{Baskakov:2002}, expressed in terms of Cartesian products of cones and cohomology with compact supports
instead of joins of simplices and reduced cohomology.

\medbreak

For a real moment-angle complex we again have a direct sum decomposition
\begin{equation}
  \BR(\Sigma) = \bigoplus_{\alpha\subset[m]} \hB^{*}(\Sigma_{\alpha}),
\end{equation}
hence also a Hochster formula
\begin{equation}
  \label{eq:Hochster-RZK}
  H^{*}(\RZK(\Sigma)) = \bigoplus_{\alpha\subset[m]} \Hc^{*}(\CU \Sigma_{\alpha})
  = \bigoplus_{\alpha\subset[m]} \tilde H^{*-1}(\Sigma_{\alpha}).
\end{equation}
Note that there are no degree shifts by~\(\deg{\alpha}\) this time.
The isomorphism becomes multiplicative
if one uses the following generalization of the product~\eqref{eq:cross-prod}.

Recall that for any open subsets~\(U=X\setminus A\) and~\(V=X\setminus B\) of a compact Hausdorff space~\(X\) there is a cup product
\begin{multline}
  \qquad
  \Hc^{*}(U) \otimes \Hc^{*}(V) = H^{*}(X,A) \otimes H^{*}(X,B) \\
  \xrightarrow{\;\;\cup\;\;}
  H^{*}\bigl(X, A\cup B\bigr) = \Hc^{*}(U\cap V),
  \qquad
\end{multline}
\cf~also~\cite[Sec.~7.4]{Massey:1978}.

Now let \(\pi_{\alpha}\colon\RZK(\Sigma_{\alpha\cup\beta})\to\RZK(\Sigma_{\alpha})\) be the (proper) projection
that sends the coordinates~\(z_{i}\) with~\(i\notin\alpha\) to the basepoint~\(e\),
and analogously for~\(\pi_{\beta}\colon\RZK(\Sigma_{\alpha\cup\beta})\to\RZK(\Sigma_{\beta})\).
We define the \(*\)-product as the composition
\begin{multline}
  \label{eq:cross-prod-general}
  \Hc^{*}(\CU \Sigma_{\alpha}) \otimes \Hc^{*}(\CU \Sigma_{\beta})
  \xrightarrow{\pi_{\alpha}^{*}\otimes\pi_{\beta}^{*}}
  \Hc^{*}(\pi_{\alpha}^{-1}(\CU \Sigma_{\alpha})) \otimes \Hc^{*}(\pi_{\beta}^{-1}(\CU \Sigma_{\beta})) \\
  \xrightarrow{\;\;\cup\;\;}
  \Hc^{*}\bigl(\pi_{\alpha}^{-1}(\CU \Sigma_{\alpha}) \cap \pi_{\beta}^{-1}(\CU \Sigma_{\beta})\bigr)
  = \Hc^{*}(\CU \Sigma_{\alpha\cup\beta}).
  \qquad\qquad
\end{multline}
In term of the isomorphism~\eqref{eq:iso-HcCU-HhB}, this exactly means to multiply representatives lying in~\(\hB(\Sigma_{\alpha})\)
and~\(\hB(\Sigma_{\beta})\) inside~\(\BR(\Sigma_{\alpha\cup\beta})\), which gives elements in~\(\hB(\Sigma_{\alpha\cup\beta})\).
Note that this construction reduces to the product~\eqref{eq:cross-prod} if \(\alpha\) and~\(\beta\) are disjoint.

The \(*\)-product is visibly graded commutative,
something that was not obvious from the multiplication rules~\eqref{eq:relations-s-t}.
Looking back, we can see that these asymmetric formulas arose from the non-commutativity of the Alexander--Whitney map
and the fact that only one of the two vertices of the singular \(1\)-simplex~\(x\) in~\(X=D^{1}\)
can be the basepoint~\(e\).

The product~\eqref{eq:cross-prod-general} coincides with the \(*\)-product
given by Bahri--Bendersky--Cohen--Gitler~\cite[Thm.~1.4]{BahriEtAl:2012}
because the former map can be thought of as induced by the partial diagonal
\begin{equation}
  \hat\Delta_{I}^{J,L}\colon \hat{Z}(K_{I})\to\hat{Z}(K_{J})\wedge\hat{Z}(K_{L})
\end{equation}
defined in~\cite[eq.~(1.5)]{BahriEtAl:2012} to construct the \(*\)-product.
In our notation we have \(K=\Sigma\), \(J=\alpha\),~\(L=\beta\) and~\(I=J\cup L=\alpha\cup\beta\).
For the comparison one uses that the compactly supported cohomology of
\begin{equation}
  \CU \Sigma = \ZK_{\Sigma}(X,A)\setminus\WK_{\Sigma}(X,A)
\end{equation}
with~\((X,A)=(D^{1},S^{0})\) is equal to the reduced cohomology of the quotient
\begin{equation}
  \hat{Z}(K)=\hat{Z}(K(X,A))=\ZK_{\Sigma}(X,A) \bigm/ \WK_{\Sigma}(X,A).
\end{equation}
considered in~\cite{BahriEtAl:2012}.

As remarked earlier, the product formula in~\cite{BahriEtAl:2012} is valid for general polyhedral products.
We can recover the version for complex moment-angle complexes
if we replace the pair~\(\RhZK(\Sigma) = \hZK_{\Sigma}(D^{1},S^{0})\) from~\eqref{eq:def-RhZK} by \(\hZK_{\Sigma}(D^{2},S^{1})\).
In this case, the distinction between disjoint and non-disjoint index sets~\(\alpha\) and~\(\beta\)
in Baskakov's formula is not necessary for the corresponding \(*\)-product
because two monomials of the form~\eqref{eq:def-a1-am} with overlapping index sets
always multiply to~\(0\) in~\(\BC(\Sigma)\).

\medbreak

We finally consider another description of~\(H^{*}(\ZK(\Sigma))\) in the polytopal case.
Let \(P\) be a simple polytope with \(m\)~facets, and let \(\Sigma\) be the boundary complex of the dual simplicial polytope.
For any subset~\(\alpha\subset[m]\), let \(P_{\alpha}\subset P\) be the union of the corresponding facets.

\begin{lemma}
  There is a ring isomorphism
  \begin{equation*}
    \Theta_{\alpha}\colon H^{*}(P,P_{\alpha}) \to \Hc^{*}(\CU \Sigma_{\alpha})
  \end{equation*}
  for any~\(\alpha\subset[m]\). Moreover, the diagram
  \begin{equation*}
    \begin{tikzcd}
      H^{*}(P,P_{\alpha}) \otimes H^{*}(P,P_{\beta}) \arrow{d}{\Theta_{\alpha}\otimes\Theta_{\beta}} \arrow{r}{\cup} & H^{*}(P,P_{\alpha\cup\beta}) \arrow{d}{\Theta_{\alpha\cup\beta}} \\
      \Hc^{*}(\CU \Sigma_{\alpha}) \otimes \Hc^{*}(\CU \Sigma_{\beta}) \arrow{r}{*} & \Hc^{*}(\CU \Sigma_{\alpha\cup\beta})
    \end{tikzcd}
  \end{equation*}
  commutes for all~\(\alpha\),~\(\beta\subset[m]\).
\end{lemma}

\begin{proof}
  Let \(\Sigma'\) be the barycentric subdivision of~\(\Sigma\), considered as a triangulation of~\(\partial P\).
  As a topological space, \(\Sigma_{\alpha}\) can be identified with a subcomplex of~\(\Sigma'\),
  hence \(\CB \Sigma_{\alpha}\) with a subcomplex of \(\CB \Sigma'\approx P\).
  We can also identify \(P_{\alpha}\) with the union of the closed blocks (or cells) in~\(\Sigma'\) dual to the vertices in~\(\alpha\),
  \cf~\cite[\S 64]{Munkres:84}.
  
  We claim that the canonical inclusion of pairs
  \begin{equation}
    \label{eq:map-barycentric-sudiv}
    (\CB \Sigma_{\alpha},\Sigma_{\alpha}) \to (\CB \Sigma',P_{\alpha})
  \end{equation}
  is a strong deformation retract.
  Similar to the proof of~\cite[Lemma~70.1]{Munkres:84}, we can define a strong deformation retraction
  that moves the vertex~\(v_{\sigma}\in C\,\Sigma'\) corresponding to a simplex~\(\sigma\in\Sigma\)
  to the vertex~\(v_{\sigma\cap\alpha}\in C\,\Sigma_{\alpha}\) along a straight line, which is inside~\(\sigma\) if \(\sigma\cap\alpha\ne\emptyset\).
  If \(\sigma\) has no vertex in~\(\alpha\), then \(v_{\sigma}\) is moved to the apex~\(v_{\emptyset}\) of the cone,
  and \(v_{\emptyset}\) is mapped to itself. We extend the map linearly to each simplex~\(\tau\in C\,\Sigma'\).
  If \(\tau\) is contained in~\(\sigma\in\Sigma\),
  then it is mapped to the cone over the simplex~\(\sigma\cap\alpha\in\Sigma_{\alpha}\)
  (with the empty simplex~\(\emptyset\) giving the apex).
  The deformation retraction restricts to one from~\(P_{\alpha}\) onto~\(\Sigma_{\alpha}\).
  We therefore get an isomorphism
  \begin{equation}
    \Theta_{\alpha}\colon H^{*}(P,P_{\alpha}) \to H^{*}(\CB \Sigma_{\alpha},\Sigma_{\alpha}) = \Hc^{*}(\CU \Sigma_{\alpha})
  \end{equation}
  in cohomology.

  To show that the above diagram commutes, we work on the chain level. We use simplicial chains for the left-hand side of~\eqref{eq:map-barycentric-sudiv},
  which canonically map to singular chains on the right. We choose a vertex ordering for~\(\CB \Sigma_{\alpha\cup\beta}\) such that
  all vertices smaller than the apex~\(v_{\emptyset}\) are in~\(\alpha\) and all greater ones in~\(\beta\). (Some may be in both.)
  To a simplex~\(\sigma\in \CB \Sigma_{\alpha\cup\beta}\)
  we have to apply the Alexander--Whitney diagonal and possibly the projections from~\(\CB \Sigma_{\alpha\cup\beta}\) to~\(\CB \Sigma_{\alpha}\)
  and~\(\CU \Sigma_{\beta}\), which send ``superfluous'' vertices to~\(v_{\emptyset}\). Afterwards we evaluate the resulting tensor product on~\(a\otimes b\)
  where \(a\),~\(b\in C(P)\) are cocycles vanishing on~\(P_{\alpha}\) and~\(P_{\beta}\), respectively.
  
  Because of the way we have ordered the simplices, the following happens:
  If \(\sigma\) does not contain \(v_{\emptyset}\), then the result is \(0\) for both ways of going through the diagram.
  Otherwise we obtain \((-1)^{\deg{b}\deg{\sigma'}}\,a(\sigma')\,b(\sigma'')\) for both ways where \(\sigma'\)
  is the front face of~\(\sigma\) ending in~\(v_{\emptyset}\) and \(\sigma''\) the back face starting there.
  Hence the diagram commutes in either case.
\end{proof}

As a consequence, we get a ring isomorphism
\begin{equation}
  H^{*}(\RZK(P)) = \bigoplus_{\alpha\subset[m]} H^{*}(P,P_{\alpha})
\end{equation}
where the multiplication on the right-hand side is given by the cup products
\begin{equation}
  H^{*}(P,P_{\alpha}) \otimes H^{*}(P,P_{\beta}) \to H^{*}(P,P_{\alpha\cup\beta})
\end{equation}
for all~\(\alpha\),~\(\beta\subset[m]\).
This description of the cohomology ring of a real moment-angle manifold
was stated without proof by Gitler and López de Medrano~\cite[p.~1526]{GitlerLopezDeMedrano:2013}.\footnote{%
In~\cite[p.~489]{LopezDeMedrano:2021}, López de Medrano writes that
``at the end [of~\cite{GitlerLopezDeMedrano:2013}] we announced, prematurely,
a formula for the cohomology ring of any~[\(\RZK(P)\)], but the proof ran into some technical problems''.}

The Alexander-dual description for moment-angle manifolds,
\begin{equation}
  H^{*}(\ZK(\Sigma)) = \bigoplus_{\alpha\subset[m]} \tilde H_{d+m-\deg{\alpha}-*}(P_{\alpha})
\end{equation}
where~\(d=\dim P-1\), has been provided by Bosio--Meersseman~\cite[Thm.~10.1]{BosioMeersseman:2006},
with the product given up to sign by the intersection products
\begin{equation}
  \tilde H_{d-k}(P_{\alpha}) \otimes \tilde H_{d-l}(P_{\beta})
  \to \tilde H_{d-(k+l)}(P_{\alpha\cap\beta})
\end{equation}
for~\(\alpha\),~\(\beta\subset[m]\) with~\(\alpha\cup\beta=[m]\) and~\(k\),~\(l\ge0\).

\end{document}